\documentclass[a4paper]{amsart}

\usepackage{amsmath,amssymb,amsthm}
\usepackage[all]{xy}
\usepackage{eucal,here}
\usepackage[dvipdfmx]{graphicx}
\usepackage[dvipdfmx,colorlinks=true,backref=page]{hyperref}

\newtheorem{theorem}{Theorem}[section]
\newtheorem{proposition}[theorem]{Proposition}
\newtheorem{lemma}[theorem]{Lemma}
\newtheorem{corollary}[theorem]{Corollary}

\theoremstyle{definition}
\newtheorem{definition}[theorem]{Definition}

\newtheorem{problem}[theorem]{Problem}

\theoremstyle{remark}
\newtheorem{remark}[theorem]{Remark}

\numberwithin{equation}{section}

\newcommand{\N}{\mathbb{N}}
\newcommand{\Z}{\mathbb{Z}}

\newcommand{\R}{\mathbb{R}}

\newcommand{\TC}{\mathsf{TC}}

\newcommand{\cat}{\mathsf{cat}}
\newcommand{\cd}{\mathsf{cd}}
\newcommand{\secat}{\mathsf{secat}}
\newcommand{\wgt}{\mathsf{wgt}}

\SelectTips{cm}{}

\title[Topological complexity sequences of groups]{Topological complexity sequences of groups}

\author{Daisuke Kishimoto}
\address{Faculty of Mathematics, Kyushu University, Fukuoka, 819-0395, Japan}
\email{kishimoto@math.kyushu-u.ac.jp}

\author{Yuki Minowa}
\address{Faculty of Mathematics, Kyushu University, Fukuoka, 819-0395, Japan}
\email{y-minowa@math.kyushu-u.ac.jp}

\date{\today}



\begin{document}

\maketitle

\begin{abstract}
  We define the topological complexity sequence of a group as the sequence of topological complexities of its Milnor constructions. This sequence may be regarded as an intrinsic refinement of the topological complexity of a group and, unlike topological complexity itself, is meaningful for groups of infinite cohomological dimension. We show that the topological complexity sequence of every group of infinite cohomological dimension is weakly increasing and unbounded. We then estimate its growth and determine its asymptotic behavior for a finite group of even order.
\end{abstract}


\section{Introduction}\label{Introduction}

Topological complexity is a numerical homotopy invariant of a space, introduced by Farber \cite{F}, which quantifies instability in motion planning problems in a continuous setting. Since its introduction, topological complexity and its variants have been studied extensively. We now recall the definition. Let $X$ be a space. The \emph{topological complexity} $\TC(X)$ is defined as the minimal integer $n$ for which there exists an open cover $X\times X=U_0\cup\cdots\cup U_n$ such that the evaluation fibration
\begin{equation}
  \label{free fibration}
  X^{[0,1]}\to X\times X,\quad\ell\mapsto(\ell(0),\ell(1))
\end{equation}
admits a section over $U_i$ for each $i=0,\ldots,n$. If no such open cover exists, we set $\TC(X)=\infty$. This definition admits a natural interpretation in terms of motion planning. A path in $X$ represents a motion of a robot in the space $X$. Accordingly, the path space $X^{[0,1]}$ can be viewed as the space of all possible motions, and a section of the above fibration corresponds to a motion planner. The number of open sets required to cover $X\times X$ and admit local sections thus reflects the instability of motion planners in $X$.

Let $G$ be a (discrete) group. As in the case of the Lusternik-Schnirelmann category of a group, we define the topological complexity of $G$ by
\[
  \TC(G)=\TC(BG),
\]
where $BG$ denotes the classifying space of $G$. This is well defined by homotopy invariance. The topological complexity of a group is of particular interest and has been intensively studied in several contexts \cite{D1,D2,DS,EFMO,FGLO1,FGLO2,FM,Gr1,Gr2,IM,K1,K2,M}. However, $\TC(G)=\infty$ whenever $G$ has infinite cohomological dimension, which is less meaningful. To address this issue, we introduce a new invariant of a group $G$ that is related to $\TC(G)$ and is meaningful for every group $G$. For $n\ge 0$, let $E_nG$ denote the $(n+1)$-fold join of $G$. The group $G$ acts diagonally on $E_nG$, and this action is free. The $n$-th Milnor construction of $G$ is defined by
\[
  B_nG=E_nG/G.
\]
Hence, there is a sequence of spaces $*=B_0G\subset B_1G\subset B_2G\subset\cdots$, and the classifying space $BG$ is given as the colimit of this sequence. For $n\ge 1$, we define
\[
  \TC^{n}(G)=\TC(B_nG).
\]
By \cite[Theorem 4]{F} together with $\dim B_nG=n$, one has $\TC^{n}(G)\le 2n$ for all $n\ge 1$. Hence $\TC^{n}(G)$ is finite for every group $G$.

\begin{definition}
  The \emph{topological complexity sequence} of a group $G$ is $\{\TC^{n}(G)\}_{n\ge 1}$.
\end{definition}

Topological complexity $\TC(X)$ is closely related to the Lusternik-Schnirelmann category $\cat(X)$, and $\cat(X)$ can be defined in terms of the $A_\infty$-structure of $X$ \cite{G,I}. It is therefore a fundamental problem to clarify the relationship between $\TC(X)$ and the $A_\infty$-structure of $X$. For a group $G$, the $A_\infty$-structure of $BG$ is given by the Milnor constructions $B_nG$. From this viewpoint as well, topological complexity sequences play an important role.

In this paper, we study the growth of the topological complexity sequences of a finite group. We begin by considering monotonicity. Farber, Tabachnikov and Yuzvinsky \cite{FTY} proved that for $n\ne 1,3,7$, $\TC(\R P^n)$ equals the immersion dimension of $\R P^n$. By \cite{FTY,J}, $\TC(\R P^n)$ for $n\le 13$ is given as follows, where our topological complexity is the one in \cite{FTY} minus one.

\renewcommand{\arraystretch}{1.2}

\begin{table}[htbp]
  \centering
  \begin{tabular}{c|ccccccccccccc}
    $n$&$1$&$2$&$3$&$4$&$5$&$6$&$7$&$8$&$9$&$10$&$11$&$12$&$13$\\\hline
    $\TC(\R P^n)$&$1$&$3$&$3$&$7$&$7$&$7$&$7$&$15$&$15$&$16$&$16$&$18$&$22$
  \end{tabular}
\end{table}

\noindent In particular, $\TC(\R P^n)$ is weakly increasing with respect to $n$ for $n\le 13$. On the other hand, the immersion dimension of $\R P^n$ is weakly increasing with respect to $n$ and unbounded. Thus since $B_n\Z_2=\R P^n$, the topological complexity sequence of $\Z_2$ is weakly increasing and unbounded. We show that this property holds for every group of infinite cohomological dimension, without using geometric methods.

\begin{theorem}
  \label{main 1}
  The topological complexity sequence of every group of infinite cohomological dimension is weakly increasing and unbounded.
\end{theorem}

\begin{remark}
  Let $G$ be a group of infinite cohomological dimension. By Theorem \ref{main 1}, we may regard $\TC^n(G)$ as approaching $\TC(G)=\infty$ as $n\to\infty$. This behavior fails in general for groups of finite cohomological dimension. By definition, $B_1\Z$ is a wedge of infinitely many copies of $S^1$, and for $n\ge 2$, $H^*(B_n\Z)$ is nontrivial only in degrees $0,1$ and $n$. Thus, using zero-divisors cup-length estimate \cite{F}, we obtain $\TC^n(\Z)\ge 2$ for all $n\ge 1$. On the other hand, by \cite{F}, $\TC(\Z)=1$. Hence, $\TC^n(\Z)$ does not approach $\TC(\Z)$ as $n\to\infty$.
\end{remark}

Let $G$ be a group of infinite cohomological dimension. For $n\ge 1$, define the growth function of the topological complexity sequence of $G$ by
\[
  \alpha_G\colon\N\to\Z,\quad n\mapsto|\{k\ge 1\mid\TC^k(G)\le n\}|.
\]
By Corollary \ref{TC^n} below, for every $n\ge 1$, one has
\begin{equation}
  \label{TC^n estimate}
  n\le\TC^n(G)\le 2n.
\end{equation}
It follows that $\alpha_G$ is well defined and satisfies
\begin{equation}
  \label{growth estimate}
  \left[\frac{n}{2}\right]\le\alpha_G(n)\le n,
\end{equation}
where $[x]$ denotes the integer part of $x$. By Theorem \ref{main 1}, $\alpha_G$ is weakly increasing. We next provide a sharper estimate of $\alpha_G$ when $G$ has even order, where every finite group has infinite cohomological dimension. For $n\ge 1$, let $\alpha(n)$ denote the sum of the digits in the dyadic expansion of $n$, and define
\[
  \beta(n)=\min\left\{2\left(\left[\frac{k}{4}\right]+\alpha\left(\left[\frac{k}{2}\right]\right)+\alpha\left(\left[\frac{k}{4}\right]+\alpha\left(\left[\frac{k}{2}\right]\right)\right)\right)-3\,\middle|\,k=n+1,\ldots,2n\right\}.
\]

\begin{theorem}
  \label{main 2}
  Let $G$ be a finite group of even order. Then for all $n\ge 1$,
  \[
    \left[\frac{n}{2}\right]\le\alpha_G(n)\le\beta(n).
  \]
\end{theorem}

Consequently, we can determine the asymptotic behavior of the growth function $\alpha_G$ when $G$ has even order.

\begin{corollary}
  \label{main 3}
  Let $G$ be a finite group of even order. Then
  \[
    \lim_{n\to\infty}\frac{\alpha_G(n)}{n}=\frac{1}{2}.
  \]
\end{corollary}

In Section \ref{Groups with nontrivial center}, using the sectional category weight introduced by Farber and Grant \cite{FG}, we derive a lower bound for $\TC^n(G)$ when $G$ contains a central $p$-subgroup. However, unlike the case of finite groups with even order, this bound does not determine the asymptotic behavior of the growth function $\alpha_G(n)$.


\subsection*{Acknowledgement}

The authors are grateful to Norio Iwase and John Oprea for useful comments, especially concerning Proposition \ref{cat(B_nG)}. They also thank the anonymous referee for pointing out an error in an earlier version of the paper. The authors were supported by JSPS KAKENHI Grant Numbers JP22K03284 (Kishimoto) and JP26KJ0241 (Minowa).


\section{Monotonicity}

In this section, we prove Theorem \ref{main 1} by using Lusternik-Schnirelmann category.

Let $X$ be a space. Recall that the \emph{Lusternik-Schnirelmann category} $\cat(X)$ is defined as the minimal integer $n$ such that there exists an open cover $X=U_0\cup\cdots\cup U_n$, where the inclusion $U_i\to X$ is null-homotopic for each $i=0,\ldots,n$. If no such open cover exists, we set $\cat(X)=\infty$. By \cite[Theorem 5]{F}, if $X$ is path-connected and paracompact, then
\begin{equation}
  \label{TC-cat}
  \cat(X)\le\TC(X)\le 2\cdot\cat(X).
\end{equation}
We remark that our normalization of $\TC(X)$ and $\cat(X)$ differs from \cite{F}. Specifically, our normalization is the one used in \cite{F} minus one, which has become standard in recent work.

Let $G$ be a group. Recall that the \emph{cohomological dimension} of $G$, denoted by $\cd(G)$, is defined as the largest integer $n$ such that $H^n(G;M)\ne 0$ for some $G$-modules $M$. It is known \cite{B,DR,S} that there exists a cohomology class $\mathfrak{b}\in H^1(BG;I)$, called the Berstein class, with the property that $\cd(G)$ coincides with the largest integer $n$ for which the $n$-th power $\mathfrak{b}^n\in H^n(BG;I^{\otimes n})$ is nontrivial. Here, $I$ denotes the augmentation ideal of the group ring $\Z G$, regarded as twisted coefficients via the natural $G$-action. In particular, if $\cd(G)=\infty$, then $\mathfrak{b}^n\ne 0$ for all $n\ge 1$.

\begin{proposition}
  \label{cat(B_nG)}
  Let $G$ be a group. Then for every $n\ge 1$,
  \[
    \cat(B_nG)=\min\{n,\cd(G)\}.
  \]
\end{proposition}

\begin{proof}
  Suppose that $\cd(G)\ge n$. Let $\mathfrak{b}\in H^1(BG;I)$ be the Berstein class. Then its $n$-th power $\mathfrak{b}^n\in H^n(BG;I^{\otimes n})$ is nontrivial. The natural map $i\colon B_nG\to BG$ is an $n$-equivalence, and hence the induced map $i^*\colon H^n(BG;I^{\otimes n})\to H^n(B_nG;I^{\otimes n})$ is injective. It follows that $i^*(\mathfrak{b})^n\ne 0$, and therefore the cup-length of $B_nG$ is at least $n$. Since the Lusternik-Schnirelmann category is bounded below by the cup-length, we obtain $\cat(B_nG)\ge n$. On the other hand, $\cat(B_nG)\le\dim B_nG=n$, and hence
  \[
    \cat(B_nG)=n
  \]
  for $\cd(G)\ge n$. By \cite[Corollary 1.3]{C}, one has $\cat(B_nG)=\cat(G)$ for $\cat(G)<n$, and Eilenberg and Ganea proved in \cite{EG} that $\cat(G)=\cd(G)$. Hence
  \[
    \cat(B_nG)=\cd(G)
  \]
  for $\cd(G)<n$.
\end{proof}

\begin{corollary}
  \label{TC^n}
  Let $G$ be a group. Then for every $n\ge 1$,
  \[
    \min\{n,\cd(G)\}\le\TC^n(G)\le 2\min\{n,\cd(G)\}.
  \]
\end{corollary}

\begin{proof}
  The claim follows from Proposition \ref{cat(B_nG)} together with \eqref{TC-cat}.
\end{proof}

\begin{lemma}
  \label{Ganea fibration}
  For every group $G$ and $n\ge 0$, the sequence $E_nG\to B_nG\to BG$ is equivalent to the $n$-th Ganea fibration of $BG$.
\end{lemma}

\begin{proof}
  The case $n=0$ is clear. Let $p\colon E\to B$ be a $G$-covering with $G$-action $\mu\colon E\times G\to E$. Recall from \cite{DL} that the Dold-Lashof construction $\mathrm{DL}(p)\colon\mathrm{DL}(E)\to\mathrm{DL}(B)$ is defined by
  \[
    \mathrm{DL}(E)=E\cup_\mu(CE\times G)\quad\text{and}\quad\mathrm{DL}(B)=B\cup_pCE,
  \]
  where $\mathrm{DL}(p)$ is obtained by gluing $p$ and the first projection $CE\times G\to CE$. The map $\mathrm{DL}(p)\colon\mathrm{DL}(E)\to\mathrm{DL}(B)$ is again a $G$-covering, and by \cite[Appendix B]{Ga}, $\mathrm{DL}(E)$ is $G$-homeomorphic to $E\star G$, where $E\star G$ denotes the join of $E$ and $G$ equipped with the diagonal $G$-action. Hence, $E_nG\to B_nG$ is isomorphic to the $n$-fold Dold-Lashof construction applied to the $G$-covering $G\to *$. In particular, there is a homotopy cofibration $E_nG\to B_nG\to B_{n+1}G$. On the other hand, there is a homotopy pullback
  \[
    \xymatrix{
      E_{n+1}G\ar[r]\ar[d]&EG\ar[d]\\
      B_{n+1}G\ar[r]&BG.
    }
  \]
  Since $EG$ is contractible, it follows that the sequence $E_{n+1}G\to B_{n+1}G\to BG$ is a homotopy fibration. Therefore the claim follows by induction on $n$.
\end{proof}

\begin{remark}
  In \cite[Appendix B]{Ga}, it is assumed that all topological groups are countable CW-groups. This assumption is made to simplify the treatment of product spaces, but it is unnecessary for discrete groups. Therefore, the result of \cite[Appendix B]{Ga} applies to arbitrary discrete groups.
\end{remark}

We now prove Theorem \ref{main 1}.

\begin{proof}
  [Proof of Theorem \ref{main 1}]
  Let $G$ be a group of infinite cohomological dimension. By Corollary \ref{TC^n}, the topological complexity sequence of $G$ is unbounded. By Lemma \ref{Ganea fibration}, there is a homotopy cofibration
  \[
    E_nG\to B_nG\to B_{n+1}G.
  \]
  Since $E_nG$ is homotopy equivalent to a wedge of $n$-spheres, the Blakers-Massey theorem implies that the natural map $B_nG\to B_{n+1}G$ is an $n$-equivalence. Garc\'{i}a Calcines and Vandembroucq \cite[Theorem 3]{GCV} showed that for $(q-1)$-connected spaces $X$ and $Y$ with $q\ge 1$, if there exists an $r$-equivalence $X\to Y$ and
  \[
    2\dim X\le r+q\TC(Y)-1,
  \]
  then $\TC(X)\le\TC(Y)$. By Corollary \ref{TC^n}, $\TC^{n+1}(G)\ge n+1$, and therefore
  \[
    2\dim B_nG=2n\le n+\TC^{n+1}(G)-1.
  \]
  Hence we can apply the above result to the case $X=B_nG$ and $Y=B_{n+1}G$ with $(q,r)=(1,n)$, and conclude that $\TC^n(G)\le\TC^{n+1}(G)$.
\end{proof}


\section{Estimate of the growth}

In this section, we estimate the growth function $\alpha_G(n)$ for a nontrivial group $G$ of even order using $\mathcal{D}$-topological complexity.

We begin by recalling $\mathcal{D}$-topological complexity, introduced by Farber, Grant, Lupton, and Oprea \cite{FGLO1,FGLO2}.

\begin{definition}
  The $\mathcal{D}$-topological complexity $\TC^\mathcal{D}(X)$ of a path-connected space $X$ is defined as the minimal integer $n$ for which there exists an open cover $X\times X=U_0\cup\cdots\cup U_n$ such that for each $i=0,\ldots,n$ and every choice of basepoint $u_i\in U_i$, the natural map $\pi_1(U_i,u_i)\to\pi_1(X\times X,u_i)$ has image contained in a subgroup conjugate to the diagonal subgroup of $\pi_1(X\times X,u_i)\cong\pi_1(X)\times\pi_1(X)$. If no such open cover exists, we set $\TC^\mathcal{D}(X)=\infty$.
\end{definition}

Recall that the \emph{sectional category}, or the Schwarz genus, of a map $f\colon X\to Y$, denoted by $\secat(f)$, is defined as the minimal integer $n$ for which there exists an open cover $Y=U_0\cup\cdots\cup U_n$ such that $f$ admits a homotopy section over each $U_i$ for $i=0,\ldots,n$. Since the end point fibration $X^{[0,1]}\to X\times X$ is, up to homotopy, identified with the diagonal map $X\to X\times X$, the topological complexity $\TC(X)$ equals the sectional category of the diagonal map $X\to X\times X$. By definition, sectional category satisfies the following properties.

\begin{proposition}
  \label{secat pullback}
  \begin{enumerate}
    \item For a homotopy pullback
    \[
      \xymatrix{
        X\ar[r]\ar[d]_f&Z\ar[d]^g\\
        Y\ar[r]&W,
      }
    \]
    there is an inequality
    \[
      \secat(f)\le\secat(g).
    \]

    \item If there is a homotopy commutative diagram
    \[
      \xymatrix{
        X\ar[r]^i\ar[d]^f&Z\ar[r]^p\ar[d]^g&X\ar[d]^f\\
        Y\ar[r]^j&W\ar[r]^q&Y
      }
    \]
    in which $p\circ i\simeq 1_X$ and $q\circ j\simeq 1_Y$, then
    \[
      \secat(f)\le\secat(g).
    \]
  \end{enumerate}
\end{proposition}

For $G$-spaces $X$ and $Y$, let $X\times_GY$ denote the quotient of $X\times Y$ by the diagonal action of $G$. The following three propositions are proved in \cite{FGLO2}.

\begin{proposition}
  \label{DTC=secat}
  Let $X$ be a connected, locally path-connected and semi-locally simply-connected space with $\pi=\pi_1(X)$, and let $\widetilde{X}\to X$ be the universal covering. Then
  \[
    \TC^\mathcal{D}(X)=\secat(\widetilde{X}\times_\pi\widetilde{X}\to X\times X).
  \]
\end{proposition}

\begin{proposition}
  \label{DTC<TC}
  For a connected, locally path-connected and semi-locally simply-connected space $X$, one has
  \[
    \TC^\mathcal{D}(X)\le\TC(X).
  \]
\end{proposition}

\begin{proposition}
  \label{DTC RP}
  For every $n\ge 1$,
  \[
    \TC^\mathcal{D}(\R P^n)=\TC(\R P^n).
  \]
\end{proposition}

Although it is not known whether $\TC^\mathcal{D}(B_nG)$ equals $\TC(B_nG)$ for all groups $G$ of infinite cohomological dimension, we have the following result.

\begin{proposition}
  \label{DTC TC estimate}
  Let $G$ be a group of infinite cohomological dimension, and let $n\ge 2$. Then
  \[
    \TC^\mathcal{D}(B_nG)\ge n\quad\text{and}\quad\TC^n(G)\le\TC^\mathcal{D}(B_nG)+1.
  \]
  Moreover, if $\TC^\mathcal{D}(B_nG)\ge n+1$, then
  \[
    \TC^n(G)=\TC^\mathcal{D}(B_nG).
  \]
\end{proposition}

\begin{proof}
  There is a homotopy pullback
  \[
    \xymatrix{E_nG\ar[r]\ar[d]&E_nG\times_GE_nG\ar[d]\\
    B_nG\ar[r]&B_nG\times B_nG}
  \]
  where the top and bottom maps are the diagonal maps. Suppose that $n\ge 2$. Then $E_nG\to B_nG$ is the universal covering. Hence, by Propositions \ref{secat pullback} and \ref{DTC=secat}, together with \cite[Proposition 45]{S},
  \[
    \TC^\mathcal{D}(B_nG)\ge\secat(E_nG\to B_nG)=\cat(B_nG)=n.
  \]
  Since the universal cover $E_nG$ of $B_nG$ is $(n-1)$-connected, it follows from \cite[Theorem 7.10]{EFMO} that $\TC^n(G)\le n+1$ if $\TC^\mathcal{D}(B_nG)=n$, and that $\TC^n(G)=\TC^\mathcal{D}(B_nG)$ if $\TC^\mathcal{D}(B_nG)\ge n+1$. Thus, together with the inequality $\TC^\mathcal{D}(B_nG)\ge n$, the remaining claims follow.
\end{proof}

\begin{remark}
  Let $G$ be a group, not necessarily of infinite cohomological dimension. Since $B_1G$ is homotopy equivalent to the wedge of $|G|-1$ copies of $S^1$, it follows from \cite[Proposition 2.5]{FGLO2} that
  \[
    \TC^1(G)=\TC^\mathcal{D}(B_1G).
  \]
  Moreover, a cup-length argument shows that
  \[
    \TC^1(G)=
    \begin{cases}
      0&(|G|=1)\\
      1&(|G|=2)\\
      2&(|G|\ge 3).
    \end{cases}
  \]
\end{remark}

Let $X$ and $Y$ be path-connected spaces. It is proved in \cite{FGLO2} that if there exists a map $X\to Y$ which is an isomorphism in $\pi_1$, then
\[
  \TC^\mathcal{D}(X)\le\TC^\mathcal{D}(Y).
\]
We improve this result in a special case. Let $G$ be a group, and let $H$ be a subgroup of $G$. We consider the action of $H\times H$ on $G$ given by
\begin{equation}
  \label{action}
  (x,y)\cdot z=xzy^{-1}
\end{equation}
for $(x,y)\in H\times H$ and $z\in G$. Note that all stabilizer subgroups of this action are contained in the diagonal subgroup of $H\times H$ if and only if for $x,y\in H$ and $z\in G$, $xzy^{-1}=z$ implies $x=y$. Hence if $H$ satisfies this condition, then we say that $H$ has \emph{diagonal stabilizers}. For example, if $H$ has order two, then it has diagonal stabilizers. Indeed, for any $z\in G$ and the nontrivial element $y\in H$, we have $zy\ne z$ and $yz\ne z$. Hence the claim follows.

\begin{proposition}
  \label{DTC H<G}
  Let $G$ be a group, and let $H$ be a subgroup of $G$ having diagonal stabilizers. Then
  \[
    \TC^\mathcal{D}(B_nH)\le\TC^\mathcal{D}(B_nG).
  \]
\end{proposition}

\begin{proof}
  Consider the canonical left action of $G$ on itself. Then there is a (homotopy) pullback
  \[
    \xymatrix{
      E_nH\times_HG\ar[r]\ar[d]&E_nG\ar[d]\\
      B_nH\ar[r]&B_nG
    }
  \]
  where we identify $E_nG=E_nG\times_GG$. We now consider the right $G$-action on $E_nH\times_HG$ and form $(E_nH\times_HG)\times_G(E_nH\times_HG)$. Since the top map of the above diagram is $G$-equivariant, there is a commutative diagram
  \[
    \xymatrix{
      (E_nH\times_HG)\times_G(E_nH\times_HG)\ar[r]\ar[d]&E_nG\times_GE_nG\ar[d]\\
      B_nH\times B_nH\ar[r]&B_nG\times B_nG.
    }
  \]
  Since both the left and right vertical maps are covering maps with common fiber $G$, this diagram is a (homotopy) pullback. Hence, it follows from Propositions \ref{secat pullback} and \ref{DTC=secat} that
  \[
    \secat((E_nH\times_HG)\times_G(E_nH\times_HG)\to B_nH\times B_nH)\le\TC^\mathcal{D}(B_nG).
  \]
  Let $H\backslash G\slash H=\{x_i\}_{i\in I}$ be the double coset decomposition. Define a map $\rho\colon G\to H$ by
  \[
    \rho(ax_ib)=ab\quad(a,b\in H).
  \]
  Since $H$ has diagonal stabilizers, if $a,b,c,d\in H$ satisfy $ax_ib=cx_id$, then $ab=cd$, which implies that $\rho$ is well defined. By definition, $\rho$ is $(H\times H)$-equivariant, where $H\times H$ acts on $G$ and $H$ by \eqref{action}. Consider an $(H\times H)$-equivariant map
  \[
    \theta\colon G\times_HG\to G,\quad[(x,y)]\mapsto xy^{-1}
  \]
  which restricts to an $(H\times H)$-equivariant bijection $H\times_HH\to H$. Then the composite
  \[
    G\times_HG\xrightarrow{\theta}G\xrightarrow{\rho}H\xrightarrow{\theta^{-1}}H\times_HH
  \]
  is $(H\times H)$-equivariant. This map induces an $(H\times H)$-equivariant map
  \[
    (E_nH\times_HG)\times_G(E_nH\times_HG)\to (E_nH\times_HH)\times_H(E_nH\times_HH)=E_nH\times_HE_nH
  \]
  which restricts to the identity map of $E_nH\times_HE_nH$. Hence there is a commutative diagram
  \[
    \xymatrix{
      E_nH\times_HE_nH\ar[r]\ar[d]&(E_nH\times_HG)\times_G(E_nH\times_HG)\ar[r]\ar[d]&E_nH\times_HE_nH\ar[d]\\
      B_nH\times B_nH\ar@{=}[r]&B_nH\times B_nH\ar@{=}[r]&B_nH\times B_nH
    }
  \]
  such that the composite of the top maps is the identity. Thus by Propositions \ref{secat pullback} and \ref{DTC=secat}, we obtain
  \begin{align*}
    \TC^\mathcal{D}(B_nH)&=\secat(E_nH\times_HE_nH\to B_nH\times B_nH)\\
    &\le\secat((E_nH\times_HG)\times_G(E_nH\times_HG)\to B_nH\times B_nH).
  \end{align*}
  Therefore the proof is complete.
\end{proof}

\begin{corollary}
  \label{DTC Z_2}
  If a group $G$ has even order, then for every $n\ge 1$,
  \[
    \TC^n(\Z_2)\le\TC^n(G).
  \]
\end{corollary}

\begin{proof}
  Since $G$ has even order, it has a subgroup $H$ of order two. Since $H$ has diagonal stabilizers as mentioned above, it follows from Propositions \ref{DTC<TC}, \ref{DTC RP} and \ref{DTC H<G} that
  \[
    \TC^n(\Z_2)=\TC(\R P^n)=\TC^\mathcal{D}(\R P^n)\le\TC^\mathcal{D}(B_nG)\le\TC^n(G)
  \]
  where $B_n\Z_2=\R P^n$.
\end{proof}

\begin{proof}
  [Proof of Theorem \ref{main 2}]
  As mentioned in Section \ref{Introduction}, Farber, Tabachnikov and Yuzvinsky \cite{FTY} proved that for $n\ne 1,3,7$, $\TC^n(\Z_2)$ equals the immersion dimension of $\R P^n$. On the other hand, Davis \cite{D} proved that for every $k\ge 1$, $\R P^{2(k+\alpha(k)-1)}$ cannot be immersed into $\R^{4k-2\alpha(k)}$. Hence
  \begin{equation}
    \label{Z_2 upper bound}
    \TC^{2(k+\alpha(k)-1)}(\Z_2)>4k-2\alpha(k).
  \end{equation}
  By Theorem \ref{main 1}, this implies that for every $k\ge 1$,
  \begin{equation}
    \label{alpha inequality}
    \alpha_{\Z_2}(4k-2\alpha(k))\le 2(k+\alpha(k)-1)-1.
  \end{equation}
  Let $\epsilon(n)$ denote the parity of $n$, and set $k=\frac{n-\epsilon(n)}{2}+\alpha(n)$. Note that
  \begin{equation}
    \label{alpha formula}
    \alpha(a)\le a,\quad\alpha(2a)=\alpha(a),\quad\alpha(2a+1)=\alpha(2a)+1,\quad\alpha(a+b)\le\alpha(a)+\alpha(b).
  \end{equation}
  Moreover, by Legendre's formula, one has $\nu_2(a!)=a-\alpha(a)$, where $\nu_2(a)$ denotes the dyadic valuation of $a$. Thus
  \begin{align*}
    2k-\alpha(k)&=2k-\alpha(2k)\\
    &=\nu_2((2k)!)\\
    &=\nu_2((n-\epsilon(n)+2\alpha(n))!)\\
    &=n-\epsilon(n)+2\alpha(n)-\alpha(n-\epsilon(n)+2\alpha(n))\\
    &\ge n-\epsilon(n)+2\alpha(n)-(\alpha(n)-\epsilon(n)+\alpha(2\alpha(n)))\\
    &=n+\alpha(n)-\alpha(\alpha(n))\\
    &\ge n.
  \end{align*}
  Since every finite group has infinite cohomological dimension, it follows from Theorem \ref{main 1} that the growth function $\alpha_{\Z_2}$ is weakly increasing. Therefore for every $n\ge 1$,
  \begin{align*}
    \alpha_{\Z_2}(2n)&\le\alpha_{\Z_2}(4k-2\alpha(k))\\
    &\le 2(k+\alpha(k)-1)-1\\
    &=n-\epsilon(n)+2\alpha(n)+2\alpha(n-\epsilon(n)+2\alpha(n))-3.
  \end{align*}
  By Corollary \ref{DTC Z_2}, $\alpha_G(2n)\le\alpha_{\Z_2}(2n)$, and by monotonicity, $\alpha_G(2n-1)\le\alpha_G(2n)$. Combining these inequalities, we obtain
  \[
    \alpha_G(n)\le f(n+1),
  \]
  where
  \[
    f(n)=2\left(\left[\frac{n}{4}\right]+\alpha\left(\left[\frac{n}{2}\right]\right)+\alpha\left(\left[\frac{n}{4}\right]+\alpha\left(\left[\frac{n}{2}\right]\right)\right)\right)-3.
  \]
  Since $\alpha(k)\le[\log_2k]+1$, it follows from \eqref{alpha formula} that
  \begin{align}
    \label{f upper bound}
    f(n)&\le 2\left(\left[\frac{n}{4}\right]+\alpha\left(\left[\frac{n}{2}\right]\right)+\alpha\left(\left[\frac{n}{4}\right]\right)+\alpha\left(\alpha\left(\left[\frac{n}{2}\right]\right)\right)\right)-3 \\
    &\le 2\left(\left[\frac{n}{4}\right]+2\alpha\left(\left[\frac{n}{2}\right]\right)+\alpha\left(\left[\frac{n}{4}\right]\right)\right)-3 \notag\\
    &\le 2\left(\left[\frac{n}{4}\right]+2\log_2\left[\frac{n}{2}\right]+2+\log_2\left[\frac{n}{4}\right]+1\right)-3 \notag\\
    &\le 2\left(\frac{n}{4}+2\log_2\frac{n}{2}+\log_2\frac{n}{4}+3\right)-3 \notag\\
    &=\frac{n}{2}+6\log_2n-5. \notag
  \end{align}
  Since $f(2n)\ge n$, we obtain $f(n)\le f(2n)$ for $n\ge 8$. By Table \ref{f and beta} below, we also have $f(n)\le f(2n)$ for $n\le 7$. Therefore by monotonicity of $\alpha_G(n)$, the proof is complete.
\end{proof}

\begin{remark}
  By definition, $\beta(n)\le f(n+1)$. This inequality is strict for infinitely many values of $n$.

  One can see that $\beta(n)$ is actually sharper than $f(n+1)$ from the following table.
  \begin{table}[htbp]
    \centering
    \caption{}\label{f and beta}
    \begin{tabular}{c|ccccccccccccc}
      $n$&$1$&$2$&$3$&$4$&$5$&$6$&$7$&$8$&$9$&$10$&$11$&$12$&$13$\\\hline
      $f(n+1)$&$1$&$1$&$3$&$3$&$7$&$7$&$7$&$7$&$7$&$7$&$11$&$11$&$13$\\
      $\beta(n)$&$1$&$1$&$3$&$3$&$7$&$7$&$7$&$7$&$7$&$7$&$11$&$11$&$11$
    \end{tabular}
  \end{table}
\end{remark}

\begin{proof}
  [Proof of Corollary \ref{main 3}]
  By \eqref{growth estimate} and \eqref{f upper bound}, we have
  \[
    \frac{n-1}{2}\le\alpha_G(n)\le\frac{n}{2}+6\log_2n-5.
  \]
  Therefore the claim follows.
\end{proof}


\section{Groups with central $p$-subgroup}\label{Groups with nontrivial center}

In this section, using sectional category weight, we provide a lower bound for $\TC^n(G)$ when $G$ contains a central $p$-subgroup.

Recall from \cite{FG} that the \emph{sectional category weight} $\wgt_f(u)$ of a cohomology class $u\in H^*(Y)$ with respect to a map $f\colon X\to Y$ is defined as the maximal integer $n$ such that for any map $g\colon Z\to Y$ satisfying $\secat(g^*f\colon W\to Z)<n$, one has $g^*(u)=0$. Here $g^*f\colon W\to Z$ denotes the homotopy pullback of $f$ along $g$. If no such integer exists, we set $\wgt_f(u)=\infty$. The choice of coefficients for cohomology is arbitrary. Note that if $i\colon x_0\to X$ is the basepoint inclusion of a path-connected space $X$, then $\wgt_i(x)$ coincides with the classical category weight. The sectional category weight satisfies the following properties \cite{FG}.

\begin{proposition}
  \label{weight}
  Let $f\colon X\to Y$ be a map, and let $u\in H^*(Y)$ be nontrivial.

  \begin{enumerate}
    \item There are inequalities
    \[
      \wgt_f(u)\le|u|\quad\text{and}\quad\wgt_f(u)\le\secat(f).
    \]

    \item For any map $g\colon Z\to Y$,
    \[
      \wgt_f(u)\le\wgt_{g^*f}(g^*(u)).
    \]

    \item If $u=u_1\cdots u_n$, then
    \[
      \wgt_f(u)\ge\wgt_f(u_1)+\cdots+\wgt_f(u_n).
    \]
  \end{enumerate}
\end{proposition}

Let $\Delta_X\colon X\to X\times X$ denote the diagonal map. 


\begin{lemma}
  \label{weight naturality}
  Let $f\colon X\to Y$ be a map, and let $u\in H^*(Y\times Y)$ be nontrivial. Then
  \[
    \wgt_{\Delta_Y}(u)\le\wgt_{\Delta_X}((f\times f)^*(u)).
  \]
\end{lemma}

\begin{proof}
  Let $g\colon Z\to W$ be a map, and let $g(k)\colon Z(k)\to W$ denote its $k$-fold fiberwise join. By \cite[Proposition 1]{Gr3}, for every nontrivial class $x\in H^*(W)$, we have $\wgt_g(x)\ge k$ if and only if $g(k)^*(x)=0$. Hence it suffices to show that for all $k\ge 1$, if $\Delta_Y(k)^*(u)=0$, then $\Delta_X(k)^*((f\times f)^*(u))=0$. Indeed, this immediately follows from the commutative diagram
  \[
    \xymatrix{
      X(k)\ar[r]^{\tilde{f}}\ar[d]_{\Delta_X(k)}&Y(k)\ar[d]^{\Delta_Y(k)}\\
      X\times X\ar[r]^{f\times f}&Y\times Y,
    }
  \]
  which shows that $\Delta_X(k)^*((f\times f)^*(u))=\tilde{f}^*(\Delta_Y(k)^*(u))$.
\end{proof}

In this section, we apply Proposition \ref{DTC H<G} to obtain an upper bound for the growth function $\alpha_G(n)$ when $G$ is a $p$-group.

\begin{lemma}
  \label{TC Z_p}
  Let $G$ be a group with central $p$-subgroup. Then for all $n\ge 1$,
  \[
    \TC^n(\Z_p)-1\le\TC^n(G).
  \]
\end{lemma}

\begin{proof}
  Since $G$ contains a central $p$-subgroup, it contains a central subgroup $H$ isomorphic to $\Z_p$. Every central subgroup of $G$ has diagonal stabilizer. Hence, by Propositions \ref{DTC<TC}, \ref{DTC TC estimate} and \ref{DTC H<G}, we obtain
  \[
    \TC^n(H)-1\le\TC^\mathcal{D}(B_nH)\le\TC^\mathcal{D}(B_nG)\le\TC^n(G).
  \]
  This proves the claim.
\end{proof}

\begin{proposition}
  [{cf. \cite[Theorem 11]{FG2}}]
  \label{Z_p upper bound}
  For any integers $0\le k,l\le n$ with $\binom{k+l}{k}\not\equiv 0\mod p$, we have
  \[
    \TC^{2n+1}(\Z_p)\ge 2(k+l)+1.
  \]
\end{proposition}

\begin{proof}
  The claim is trivial for $n=0$, so we assume $n\ge 1$. Let $L$ denote the $(2n+1)$-dimensional mod $p$ lens space. Its mod $p$ cohomology is given by
  \[
    H^*(L;\Z_p)=\varDelta(u)\otimes\Z_p[v]/(v^{n+1}),\quad\beta u=v,
  \]
  where $|u|=1$ and $|v|=2$. Here, $\varDelta(u)=\Lambda(u)$ for $p$ odd, while for $p=2$, $\varDelta(u)$ is isomorphic to $\Lambda(u)$ as a module, with $u^2=v$. It is shown in \cite{FG2} that
  \[
    \wgt_{\Delta_L}(v\times 1-1\times v)\ge 2.
  \]
  Regard $B\Z_p$ as the infinite dimensional lens space. Then $L$ is the $(2n+1)$-skeleton of $B\Z_p$, and consequently the inclusion $i\colon B_{2n+1}\Z_p\to B\Z_p$ factors as the composite
  \[
    B_{2n+1}\Z_p\xrightarrow{f}L\xrightarrow{j}B\Z_p,
  \]
  up to homotopy, where the map $j$ is the inclusion. The map $j$ induces an isomorphism in mod $p$ cohomology of dimension $\le 2n+1$. By Lemma \ref{Ganea fibration}, there is a homotopy fibration
  \[
    E_{2n+1}\Z_p\to B_{2n+1}\Z_p\xrightarrow{i}B\Z_p.
  \]
  Since $E_{2n+1}\Z_p$ is homotopy equivalent to a wedge of $S^{2n+1}$, it follows that $i$ is an $(2n+1)$-equivalence. In particular, $i$ induces an injection in mod $p$ cohomology of dimension $\le 2n+1$. Since both $L$ and $B_{2n+1}\Z_p$ are $(2n+1)$-dimensional, it follows that the map $f$ induces an injection in mod $p$ cohomology. Therefore $f^*(v)\ne 0$, and by Lemma \ref{weight naturality},
  \[
    \wgt_{\Delta_{B_{2n+1}\Z_p}}(f^*(v)\times 1-1\times f^*(v))\ge\wgt_{\Delta_L}(v\times 1-1\times v)\ge 2.
  \]
  Now suppose that $0\le k,l\le n$ satisfy $\binom{k+l}{k}\not\equiv 0\mod p$. Let
  \[
    x=(f^*(u)\times 1-1\times f^*(u))(f^*(v)\times 1-1\times f^*(v))^{k+l}.
  \]
  Then $x$ contains the term $(-1)^l\binom{k+l}{l}f^*(u)f^*(v)^k\times f^*(v)^l$, which is nonzero since $f^*\colon H^*(L;\Z_p)\to H^*(B_{2n+1}\Z_p;\Z_p)$ is injective. Hence $x\ne 0$, and therefore, by Lemma \ref{weight},
  \[
    \wgt_{\Delta_{B_{2n+1}\Z_p}}(x)\ge 2(k+l)+1.
  \]
  Thus, applying Proposition \ref{weight}, the claim follows.
\end{proof}

\begin{remark}
  Let $k=a_0+a_1p+\cdots+a_mp^m$ and $l=b_0+b_1+\cdots+b_mp^m$ be the $p$-adic expansions of nonnegative integers $k$ and $l$. Since $0\le a_i+b_i\le 2p-2$ and $a_i+b_i-p<b_i$ for all $i$, Lucas's theorem implies that $\binom{k+l}{l}\equiv 0\mod p$ whenever $a_i+b_i\ge p$ for some $i$. Consequently, any choice of $k$ and $l$ violating the condition $a_i + b_i\le p-1$ for some $i$ does not contribute to a nontrivial sectional category weight argument. Hence Proposition \ref{Z_p upper bound} yields the best possible lower bound for $\TC^{2n+1}(\Z_p)$ that can be obtained using sectional category weights in the mod $p$ cohomology of a lens space.
\end{remark}

For a positive integer $n$, write its $p$-adic expansion as $n=a_0+a_1p+\cdots+a_mp^m$. Define
\[
  \delta(n)=\min\{k\mid a_i\le\tfrac{p-1}{2}\text{ for all }i\ge k\}.
\]

\begin{proposition}
  \label{p-group}
  Let $G$ be a group with central $p$-subgroup. Then for $n\ge 1$,
  \[
    \TC^{2n+1}(G)\ge 4n-2\sum_{i<\delta(n)}(2a_i-p+1)p^i
  \]
  where $n=a_0+a_1p+\cdots+a_mp^m$ is the $p$-adic expansion of $n$.
\end{proposition}

\begin{proof}
  Let $k=n$ and
  \[
    l=\sum_{i<\delta(n)}(p-1-a_i)p^i+\sum_{i\ge\delta(n)}a_ip^i.
  \]
  Clearly, $l\le n$. By Lucas's formula, $\binom{k+l}{l}\not\equiv 0\mod p$. Hence, by Proposition \ref{Z_p upper bound},
  \[
    \TC^{2n+1}(\Z_p)\ge 2(k+l)+1=4n+1-2\sum_{i<\delta(n)}(2a_i-p+1)p^i
  \]
  Since $G$ has torsion, it is of infinite cohomological dimension. Hence, by Lemma \ref{TC Z_p}, we obtain $\TC^n(G)\ge\TC^n(\Z_p)-1$. This completes the proof.
\end{proof}

For a positive integer $n$, define $\gamma(n)$ to be the smallest integer $m>\frac{n}{2}$ such that all digits in its $p$-adic expansion are even.

\begin{corollary}
  \label{gamma}
  Let $G$ be a group with central $p$-subgroup. Then
  \[
    \alpha_G(n)\le\gamma(n).
  \]
\end{corollary}

\begin{proof}
  By Proposition \ref{p-group}, we have $\TC^{\gamma(n)+1}(G)\ge 2\gamma(n)>n$. Hence, by Theorem \ref{main 1}, it follows that
  \[
    \alpha_G(n)\le 2|\{k\ge 0\mid\TC^{2k+1}(G)\le n\}|\le\gamma(n).
  \]
  This proves the claim follows.
\end{proof}

Let $G$ be a group with central $p$-subgroup. Unlike the case of finite groups of even order, Corollary \ref{gamma} does not determine the asymptotic behavior of the growth function $\alpha_G(n)$. Indeed, if $n=2p^k$ and $p$ is odd, then
\[
  \gamma(n)=n.
\]
On the other hand, if $n=3+4(p+\cdots+p^k)$ and $p\ge 5$, then
\[
  \gamma(n)=\frac{n+1}{2}.
\]
This discrepancy reflects the fact that Proposition \ref{Z_p upper bound} is, in general, coarser than \eqref{Z_2 upper bound}.


\section{Further problems}

In this section, we pose six problems concerning topological complexity sequences and their growth.

We proved in Theorem \ref{main 1} that the topological complexity sequence of every group of infinite cohomological dimension is weakly increasing. In our proof, the assumption that the group has infinite cohomological dimension is essential. However, the monotonicity of the topological complexity sequence itself may also hold for other classes of groups, which leads us to the following problem.

\begin{problem}
  Is the topological complexity sequence weakly increasing for every group?
\end{problem}

In Section \ref{Groups with nontrivial center}, we obtain a lower bound for $\TC^n(G)$ when $G$ contains a central $p$-subgroup. However, this bound does not determine the asymptotic behavior of the growth function $\alpha_G(n)$. Moreover, the method relies essentially on the existence of a central $p$-subgroup, and therefore does not apply to all finite groups of odd order. Hence, it is natural to pose the following problems.

\begin{problem}
  Is there an upper bound for the growth function $\alpha_G(n)$ for every group $G$ of odd order that is sharper than \eqref{growth estimate}?
\end{problem}

\begin{problem}
  Does Theorem \ref{main 2} hold for every nontrivial group of odd order?
\end{problem}

In this paper, we do not address the rate of the growth of topological complexity sequences; that is, we do not estimate the difference $\alpha_G(n+1)-\alpha_G(n)$. This leads us to the following problem.

\begin{problem}
  Does there exist a group $G$ of infinite cohomological dimension such that the sequence $\{\alpha_G(n+1)-\alpha_G(n)\}_{n\ge 1}$ is bounded?
\end{problem}

There are several variants of topological complexity, and it would be interesting to study the corresponding topological complexity sequences. Rudyak \cite{R} introduced sequential topological complexity, motivated by motion planning problems involving several intermediate points. The $r$-th \emph{sequential topological complexity} of a space $X$, denoted by $\TC_r(X)$, is defined as the sectional category of the diagonal map $X\to X^r$, where $X^r$ denotes the product of $r$ copies of $X$. For a group $G$, we define the $r$-th \emph{sequential topological complexity sequence} by
\[
  \TC^n_r(G)=\TC_r(B_nG).
\]
Analogously to Corollary \ref{TC^n}, we obtain the estimate
\begin{equation}
  \label{TC_r^n}
  n\le\TC_r^n(G)\le rn
\end{equation}
for all $n\ge 1$ and any group $G$ of infinite cohomological dimension. Consequently, the growth function
\[
  \alpha_G^r\colon\N\to\Z,\quad n\mapsto|\{k\ge 1\mid\TC^k_r(G)\le n\}|
\]
is well defined, and by \eqref{TC_r^n}, it satisfies the estimate
\[
  \left[\frac{n}{r}\right]\le\alpha_G^r(n)\le n.
\]
Thus we pose the following problems.

\begin{problem}
  Is there an estimate for the growth function $\alpha_G^r$ analogous to Theorem \ref{main 2} for a finite group $G$ of even order?
\end{problem}

\begin{problem}
  Does
  \[
    \lim_{n\to\infty}\frac{\alpha_G^r(n)}{n}=\frac{1}{r}
  \]
  hold for every finite group $G$ of even order?
\end{problem}


\begin{thebibliography}{99}

  \bibitem{B} I. Berstein, On the Lusternik-Schnirelmann category of Grassmannians, Math. Proc. Camb. Philos. Soc. \textbf{79} (1976), no. 1, 129-134.

  \bibitem{C} O. Cornea, Cone-length and Lusternik-Schnirelmann category, Topology \textbf{33} (1994), no. 1, 95-111.

  \bibitem{D} D.M. Davis, A strong nonimmersion theorem for real projective spaces, Ann. of Math. (2) \textbf{120} (1984), no. 3, 517-528.

  \bibitem{DL} A. Dold and R. Lashof, Principal fibrations and fibre homotopy equivalence of bundles, Illinois J. Math. \textbf{3} (1959), 285–305.

  \bibitem{D1} A. Dranishnikov, Distributional topological complexity of groups, \texttt{arXiv:2404.03041}.

  \bibitem{D2} A. Dranishnikov, On topological complexity of hyperbolic groups, Proc. Amer. Math. Soc. \textbf{148} (2020), 4547-4556.

  \bibitem{DR} A. Dranishnikov and Y. Rudyak, On the Berstein-Svarc theorem in dimension $2$, Math. Proc. Camb. Philos. Soc. \textbf{146} (2009), no. 2, 407-413.

  \bibitem{DS} A. Dranishnikov and R. Sadykov, The topological complexity of the free product, Math. Z. \textbf{293} (2019), no. 1-2, 407-416.

  \bibitem{EG} S. Eilenberg and T. Ganea, On the Lusternik-Schnirelmann category of abstract groups, Ann. of Math. \textbf{65} (1957), 517-518.

  \bibitem{EFMO} A. Espinosa Baro, M. Farber, S. Mescher and J. Oprea, Sequential topological complexity of aspherical spaces and sectional categories of subgroup inclusions, Math. Ann. \textbf{391} (2025), no. 3, 4555-4605.

  \bibitem{F} M. Farber, Topological complexity of motion planning, Discrete Comput. Geom. \textbf{29} (2003), no. 2, 211-221.

  \bibitem{FG} M. Farber and M. Grant, Symmetric motion planning, In “Topology and robotics” M. Farber et al (eds), Contemp. Math. \textbf{438} (2007), 85-104.

  \bibitem{FG2} M. Farber and M. Grant, Robot motion planning, weights of cohomology classes, and cohomology operations, Proc. Amer. Math. Soc. \textbf{136} (2008), no. 9, 3339-3349.

  \bibitem{FGLO1} M. Farber, M. Grant, G. Lupton, and J. Oprea, Bredon cohomology and robot motion planning, Algebr. Geom. Topol. \textbf{19} (2019), 2023-2059.

  \bibitem{FGLO2} M. Farber, M. Grant, G. Lupton, and J. Oprea, An upper bound for topological complexity, Topology Appl. \textbf{225} (2019), 109-125.

  \bibitem{FM} M. Farber and S. Mescher, On the topological complexity of aspherical spaces, J. Topol. Anal. \textbf{12} (2020), no. 2, 293-319.

  \bibitem{FTY} M. Farber, S. Tabachnikov, and S. Yuzvinsky, Motion Planning in projective spaces, IMRN \textbf{34} (2003), 1853-1870.

  \bibitem{Ga} P. Gajer, Geometry of Deligne cohomology,  Invent. Math. \textbf{127} (1997), no. 1, 155-207.

  \bibitem{G} T. Ganea, Lusternik-Schnirelmann category and strong category, Illinois. J. Math. \textbf{11} (1967), 417-427.

  \bibitem{GCV} J.M. Garc\'{i}a Calcines and L. Vandembroucq, Topological complexity and the homotopy cofibre of the diagonal map, Math. Z. \textbf{274} (2013), no. 1-2, 145-165.

  \bibitem{Gr3} M. Grant, Topological complexity of motion planning and Massey products, In “Algebraic Topology - Old and New: M.M. Postnikov Memorial Conference”, M. Golasi\'{n}ski et al (eds), Banach Center Publ. \textbf{85} (2009), 193-203.

  \bibitem{Gr1} M. Grant, Parametrised topological complexity of group epimorphisms, Topol. Methods Nonlinear Anal. \textbf{60} (2022), no. 1, 287-303.

  \bibitem{Gr2} M. Grant and D. Recio-Mitter, Topological complexity of subgroups of Artin's braid groups, in “Topological Complexity and Related Topics”, M. Grant, G. Lupton and L. Vandembroucq (eds), Contemp. Math. \textbf{702} (2018), 165-176.

  \bibitem{I} N. Iwase, $A_\infty$-method in Lusternik-Schnirelmann category, Topology \textbf{41} (2002), 695-723.

  \bibitem{IM} N. Iwase and Y. Miyata, Topological complexity of $S^3/Q_8$ as fibrewise L-S category, Topol. Methods Nonlinear Anal. \textbf{62} (2023), no. 1, 239-265.

  \bibitem{J} I.M. James, Euclidean models of projective spaces, Bull. London Math. Soc. \textbf{3} (1971), 257-276.

  \bibitem{K1} B. Knudsen, The topological complexity of pure graph braid groups is stably maximal, Forum Math. Sigma \textbf{10} (2022), e93.

  \bibitem{K2} B. Knudsen, On the stabilization of the topological complexity of graph braid groups, J. Appl. Comp. Topology \textbf{10} (2026), Article: 9.

  \bibitem{M} Y. Minowa, On the topological complexity of non-simply connected spaces, \texttt{arXiv:2603.09407}.

  \bibitem{R} Y.B. Rudyak, On higher analogs of topological complexity, Topology Appl. \textbf{157} (2010), no. 5, 916-920.

  \bibitem{S} A. \u{S}varc, The genus of a fibered space, Trudy Moscov. Mat. Obsc. \textbf{10}, \textbf{11} (1961 and 1962), 217-272, 99-126, (in Amer. Math. Soc. Transl. Series 2, \textbf{55} (1966)).
\end{thebibliography}
\end{document}